\documentclass{article}
\usepackage{amsmath,amssymb,amsfonts,fullpage,amsthm,color,mathrsfs,xcolor,tikz}
\usepackage{amsmath}
\usepackage{amssymb}
\usepackage{amsfonts}
\usepackage{fullpage}
\usepackage{amsthm}
\usepackage{color}
\usepackage{mathrsfs}
\usepackage{tikz}
\usepackage{graphicx}
\usepackage{dsfont}
\usepackage{parskip}
\usepackage{rotating}
\usepackage{multicol}

\usetikzlibrary{positioning}

\newtheorem{theorem}{Theorem}[section]

\newtheorem{lemma}[theorem]{Lemma}

\newtheorem{definition}[theorem]{Definition}
\theoremstyle{remark}
\newtheorem{remark}{Remark}
\newtheorem{example}{Example}

\newcommand{\R}{\mathbb{R}}

\newcommand{\N}{\mathcal{N}}

\newcommand{\E}{\mathbb{E}} 

\date{}
\begin{document}

\title{Orthogonal polynomial duality of a two--species asymmetric exclusion process}
\author{Danyil Blyschak, Olivia Burke, Jeffrey Kuan, Dennis Li, Sasha Ustilovsky, Zhengye Zhou}
\maketitle 
\abstract{We examine type D ASEP, a two--species interacting particle system which generalizes the usual asymmetric simple exclusion process. For certain cases of type D ASEP, the process does not give priority for one species over another, even though there is nontrivial interaction between the two species. For those specific cases, we prove that the type D ASEP is self--dual with respect to an independent product of $q$--Krawtchouk polynomials. 
The type D ASEP was originally constructed in \cite{KLLPZ}, using the type D quantum groups $\mathcal{U}_q(\mathfrak{so}_6)$ and  $\mathcal{U}_q(\mathfrak{so}_8)$. That paper claimed that certain states needed to be ``discarded'' in order to ensure non--negativity. Here, we also provide a more efficient argument for the same claim.

}

\section{Introduction}
The asymmetric simple exclusion process ASEP was introduced by Spitzer in 1970 \cite{Spit70}.  The ASEP can be generalized to multiple species \cite{Ligg76}. This paper considers the \textit{type D ASEP}, introduced in \cite{KLLPZ}. The state space of type $D$ ASEP consists of interacting particles of two species on a one--dimensional lattice, where at most one particle of each species may occupy a single site, and two particles may occupy a site only if they are different species. Thus, if the lattice has $L$ sites then there are $4^L$ possible configurations. Most generally, the type $D$ ASEP has three parameters $(q, n,\delta)$, where $q\in(0,1)$, $n\in \mathbb{N}$ and $\delta\in \mathbb{R}$. Roughly speaking $q$ is the asymmetry parameter, $n$ characterizes the speed of the drift, and $\delta$ quantifies the interaction between the two species of particles. When $\delta=0$ and  $n=2,3$, it was proved that the type $D$ ASEP has blocking measures and Markov self--duality which are independent copies of the single--species blocking measures and Markov self--duality. 

The main result concerns Markov duality (see also \cite{Sch97,BS,BS2,BS3,CGRS,CGRS2,KuanJPhysA,KIMRN} for type $D$  AESP with parameter $(q,2,0)$ and $ (q,3,0)$. The previous duality function of \cite{KLLPZ} was a ``triangular'' duality generalizing Schutz's duality function \cite{Sch97}.  In this paper, we produce an ``orthogonal polynomial'' duality function. Recent work on orthogonal polynomial duality functions was done in \cite{Franceschini2018SelfDualityOM, Carinci2019OrthogonalDO, GroneDual, CFG,FKZ}. More specifically, the paper \cite{CFG} proves that the $q$--Krawtchouk polynomials are duality functions for the single--species ASEP. In this paper, we will prove that the type $D$  ASEP with parameters $n=2,3$ and $\delta=0$ are orthogonal with respect to an independent product of $q$--Krawtchouk polynomials. Because these polynomials have an additional parameter $\alpha$ that is dependent on the reversible measures, they are more suitable for asymptotics than the triangular duality functions.

The Markov duality will be proved with two different methods. The first is a direct probabilistic argument, using induction on the number of lattice sites. The second is a more algebraic method, using the $*$--bialgebra structure of $\mathcal{U}_q(\mathfrak{so}_6)$. The connection between $\mathfrak{so}_6$ and type $D$ ASEP was explored in  \cite{KLLPZ}; in fact, $\mathfrak{so}_6$ is the type $D$  Lie algebra, providing the namesake for the type $D$ ASEP. In that construction, two of the six potential particle configurations were ``discarded'' in order to ensure non--negativity of the jump rates. That paper uses a computer--aided construction of a Casimir element. Here, we calculate the ``reversible measures'' that would appear if all six potential particle configurations were used. While these measures have a factorized form, they do not appear to be an independent product for any values of $(n,\delta)$.

\textbf{''CONFLICT OF INTEREST'', ''FUNDING'' AND ''DATA AVAILABILITY'' STATEMENTS} The authors would like to acknowledge support from  NSF grant DMS-2150094, as well as the Texas A\&M University Department of Mathematics and College of Science. There is no data associated to this paper and there are no conflicts of interest.



\section{Preliminary Definitions}

\subsection{Definition of the type $D$ ASEP}

The continuous-time Markov process of interest is named the type D Asymmetric Simple Exclusion Process (ASEP) with parameters $(q, n, \delta)$. There are two species (or classes) of particles, which we will call "first-class" and "second-class" and will label with $1$ and $2$ accordingly. Particle interactions take place on a one-dimensional lattice of $L$ sites $\Lambda_L = \{1, \hdots, L\}, L \in \mathbb{N}$. We denote the state space of the type D ASEP on $L$ lattice sites by $\Omega_L$. For $\eta \in \Omega_L,$ and $x \in \Lambda_L$,  let $\eta^x=(\eta^x_1,\eta^x_2)$ denote the configuration at site $x$ and  $\eta^x_1$ and $\eta^x_2$ count the number of first class and second class particles at site $x$, respectively. We denote with $\eta_1=(\eta_1^1,\ldots,\eta_1^L)$ the ``filtered'' configuration obtained by removing all second class particles from $\eta$ and $\eta_2=(\eta_2^1,\ldots,\eta_2^L)$ the ``filtered'' configuration obtained by removing all first class particles from $\eta.$

For the dynamics, we will assume closed/reflecting boundary conditions in the case of finite $L$, in that a particle that wishes to jump outside any outer lattice site is blocked from doing so. Likewise, a particle that wishes to jump to a lattice site that is already occupied by a particle of its same class is blocked from doing so. An explicit, but lengthy, description of the model can be found in \cite{KLLPZ}. In the present paper, however, only formulas are needed. 

 First, the generator for a two--site model (i.e. $L=2$) will be given below. We index the rows and columns by ordering the possible configurations lexicographically, where $0$ denotes an empty site, $1$ denotes a class 1 particle, $2$ denotes a class 2 particle, and $3$ denotes both a class 1 and class 2 particle: 
$(0, 0),
 (0, 1),
 (0, 2),
 (0, 3),
 (1, 0),
 (1, 1),
 (1, 2),
 (1, 3),
 (2, 0),
 (2, 1),
 (2, 2),
 (2, 3),
 (3, 0),
 (3, 1),
 (3, 2),
 (3, 3).$
 The generator is then a $16 \times 16$ matrix explicitly \footnote{The matrices were generated with Python code, and is available from the third or fourth authors upon request. The Python code also verified the duality result for 3 lattice sites, corresponding to $64\times 64$ matrices.} given by (for $\delta=0)$

 \arraycolsep=.5pt\def\arraystretch{1.5}
\resizebox{\linewidth}{!}{
$
\displaystyle \left[\begin{array}{cccccccccccccccc} * & 0 & 0 & 0 & 0 & 0 & 0 & 0 & 0 & 0 & 0 & 0 & 0 & 0 & 0 & 0\\0 & * &0 & 0 & q \left(q^{1 - 2 n} + q^{2 n - 1}\right) & 0 & 0 & 0 & 0 & 0 & 0 & 0 & 0 & 0 & 0 & 0\\0 & 0 & * &0 & 0 & 0 & 0 & 0 & q \left(q^{1 - 2 n} + q^{2 n - 1}\right) & 0 & 0 & 0 & 0 & 0 & 0 & 0\\0 & 0 & 0 & * &0 & 0 & 2 q^{2} + q^{2 - 2 n} - q^{4 - 2 n} & 0 & 0 & 2 q^{2} + q^{2 - 2 n} - q^{4 - 2 n} & 0 & 0 & q^{2} \left(- q^{1 - n} + q^{n - 1}\right)^{2} & 0 & 0 & 0\\0 & \frac{q^{1 - 2 n} + q^{2 n - 1}}{q} & 0 & 0 & * &0 & 0 & 0 & 0 & 0 & 0 & 0 & 0 & 0 & 0 & 0\\0 & 0 & 0 & 0 & 0 & * &0 & 0 & 0 & 0 & 0 & 0 & 0 & 0 & 0 & 0\\0 & 0 & 0 & \left(\frac{1}{q}\right)^{2 n} - \left(\frac{1}{q}\right)^{2 n - 2} + 2 & 0 & 0 & * &0 & 0 & \left(- q^{1 - n} + q^{n - 1}\right)^{2} & 0 & 0 & q^{2 n} - q^{2 n - 2} + 2 & 0 & 0 & 0\\0 & 0 & 0 & 0 & 0 & 0 & 0 & * &0 & 0 & 0 & 0 & 0 & q \left(q^{1 - 2 n} + q^{2 n - 1}\right) & 0 & 0\\0 & 0 & \frac{q^{1 - 2 n} + q^{2 n - 1}}{q} & 0 & 0 & 0 & 0 & 0 & * &0 & 0 & 0 & 0 & 0 & 0 & 0\\0 & 0 & 0 & \left(\frac{1}{q}\right)^{2 n} - \left(\frac{1}{q}\right)^{2 n - 2} + 2 & 0 & 0 & \left(- q^{1 - n} + q^{n - 1}\right)^{2} & 0 & 0 & * &0 & 0 & q^{2 n} - q^{2 n - 2} + 2 & 0 & 0 & 0\\0 & 0 & 0 & 0 & 0 & 0 & 0 & 0 & 0 & 0 & * &0 & 0 & 0 & 0 & 0\\0 & 0 & 0 & 0 & 0 & 0 & 0 & 0 & 0 & 0 & 0 & * &0 & 0 & q \left(q^{1 - 2 n} + q^{2 n - 1}\right) & 0\\0 & 0 & 0 & \frac{\left(- q^{1 - n} + q^{n - 1}\right)^{2}}{q^{2}} & 0 & 0 & \left(\frac{1}{q}\right)^{2 - 2 n} - \left(\frac{1}{q}\right)^{4 - 2 n} + \frac{2}{q^{2}} & 0 & 0 & \left(\frac{1}{q}\right)^{2 - 2 n} - \left(\frac{1}{q}\right)^{4 - 2 n} + \frac{2}{q^{2}} & 0 & 0 & * &0 & 0 & 0\\0 & 0 & 0 & 0 & 0 & 0 & 0 & \frac{q^{1 - 2 n} + q^{2 n - 1}}{q} & 0 & 0 & 0 & 0 & 0 & * &0 & 0\\0 & 0 & 0 & 0 & 0 & 0 & 0 & 0 & 0 & 0 & 0 & \frac{q^{1 - 2 n} + q^{2 n - 1}}{q} & 0 & 0 & * &0\\0 & 0 & 0 & 0 & 0 & 0 & 0 & 0 & 0 & 0 & 0 & 0 & 0 & 0 & 0 & *\end{array}\right].
$
}
where the diagonal entries are chosen so that the rows sum to $0$. 
\begin{remark}
    All the off-diagonal entries in the generator above are non-negative when $q\in (0,1)$ and $n\in\mathbb{N}$.
\end{remark}
\begin{remark}
    The jump between $(3,0)$ and $(0,3)$ has none zero rate, which means we allow two particles to jump at the same time, this is not usual in interacting particle systems. In \cite{KLLPZ}, the model was described to have the clocks located at the bonds between adjacent vertex sites, thus we can have  jumps between $(3,0)$ and $(0,3)$.
\end{remark}

Letting $\mathcal{L}$ denote the $16\times16$ matrix above, the generator in general is given by 
$$
\mathcal{L}^{1,2} + \mathcal{L}^{2,3} + \ldots +\mathcal{L}^{L-1,L}
$$
where $\mathcal{L}^{x,x+1}$ is usual notation, denoting the matrix acting on adjacent lattice sites $x$ and $x+1$. Note that when there are only species $1$ or $2$ particles, then the jump rates are
$$
q^{\pm 1} (q^{1-2n}+q^{2n-1}),
$$
making it a usual ASEP with time rescaled by $(q^{1-2n}+q^{2n-1})$. However, the parameter $n$ also appears in the interactions between the species of particles.

\subsection{Duality}
Two Markov processes $X,Y$ with  corresponding state spaces $\mathcal{X}, \mathcal{Y}$ are said to be dual with respect to a function $D : \mathcal{X} \times \mathcal{Y} \to \R$ if, for all $x \in \mathcal{X}, y \in \mathcal{Y}$, and $t\geq 0$, we have
\begin{equation}\label{Duality_Defn}
    \E_x[D(X(t),y)] = \E_y[D(x,Y(t))].
\end{equation}
When the processes in question are independent copies of each other, we refer to the above property as "self-duality."\\

 The following condition, called the "interlacing property", is equivalent to Markov duality:

\begin{equation}\label{Interweaving}
    L_XD=DL_Y^T,
\end{equation}
where $L_X$, $L_Y$ are matrix forms of the generators for the processes $X$, $Y$, and $D$ is the matrix whose entries are outputs of the duality function (all these matrices are assumed to index their rows/columns with $x \in \mathcal{X}$ and $y \in \mathcal{Y}$ in an identical manner). In the case of self-duality, the interlacing property becomes
\begin{equation}
    LD=DL^T.
\end{equation}
We note that this interlacing property is particularly useful for computationally verifying Markov duality for a fixed number of sites.

\subsection{$q$--deformed notation}\label{qdn}
 Define $q$-exponential, the $q$-analog of an exponential function, as \[\text{exp}_q(x):= \sum_{n = 0}^{\infty} \frac{x^n}{\{n\}_q!} \]
 where
\[\{n\}_q!= \prod_{k=1}^{n} \{n\}_q,  \qquad  \{n\}_q=\frac{1-q^n}{1-q}, \]
with $\{0\}_q! = 1.$ \\

 We define the q-Pochhammer symbol for $a \in \R$ and $m \in \N$  as follows:
\begin{equation}\label{QPoch_Defn}
        (a;q)_m := \prod_{k=0}^{m-1}(1-aq^k) = (1-a)(1-aq)\cdots(1-aq^{m-1}),    
\end{equation}
 as well as
\begin{equation}\label{QPoch_Inf_Defn}
        (a;q)_{\infty} := \prod_{k=0}^{\infty}(1-aq^k) = (1-a)(1-aq)(1-aq^2)\cdots,    
\end{equation}
 and also employing the notation $(a)_m := (a;q^2)_m$ and $(a)_\infty := (a;q^2)_\infty.$\\

 We can rewrite the $q$-binomial coefficient in terms of this notation:
\begin{equation}\label{QBinom_Defn}
    \binom{n}{k}_q = (-1)^k q^{k(n+1)}\frac{(q^{-2n})_k}{(q^2)_k}.
\end{equation}
Define the new $q$-exponentials
$$
e_{q}(z)=\sum_{n=0}^{\infty} \frac{z^{n}}{(q ; q)_{n}}=\frac{1}{(z ; q)_{\infty}}, \text{ for } \quad|z|<1, \quad \text { and } \quad \mathcal{E}_{q}(z)=\sum_{n=0}^{\infty} \frac{q^{\binom{n}{2}}z^{n}}{(q ; q)_{n}}=(-z ; q)_{\infty}.
$$
    
\subsection{Algebraic Definitions}

\subsubsection{The Lie Algebra $\mathfrak{so}_{2 m}$}\label{diag}

In this paper, we examine the $\mathfrak{so}_{2 m}$ Lie algebra, which is defined as the following set of matrices:
\begin{equation} 
\begin{split}
\label{Definition_so2n}
\mathfrak{so}_{2m}(\mathbb{C})=\left\{\left(\begin{array}{ll} A & B \\ 
C & D
\end{array}\right) \Bigg|A, B, C, D \in \mathbb{C}^{m \times m}, A=-D^{T}, B=B^{T}, C=C^{T}\right\}.
\end{split}
\end{equation} 

 Define the Cartan subalgebra $\mathfrak{h}$ of $\mathfrak{so}_{2m}$ as the subalgebra of diagonal matrices in $\mathfrak{so}_{2m}$. 

 Since $\mathfrak{so}_{2m}$ is a semisimple Lie algebra, $\mathfrak{g}=\mathfrak{so}_{2m}$ can be written as a direct sum
 
 $$
\mathfrak{g} = \mathfrak{h} \oplus \bigoplus_{\beta_i} \mathfrak{g_{\beta_i}}
$$
where $\beta_i$ are linear functionals corresponding to the positive simple roots of $\mathfrak{so}_{2m}$
and $\mathfrak{g_\beta}$ are the one-dimensional eigenspaces of the adjoint representation of $\mathfrak{h}$, i.e. for each $h\in\mathfrak{h}$ and $g\in\mathfrak{g_\beta}$, $[h,g] = \beta(h) g.$\\ 

 The $\mathfrak{so}_{2 m}$ Lie Algebra corresponds to a type $D_m$ Dynkin diagram. We can define a Cartan matrix to encode the $D_m$ Dynkin diagram as follows: an entry -1 in the Cartan matrix indicates a single edge between nodes i and j and an entry of 0 indicates no connection between the nodes. 
\begin{align*}
a_{i j}= \begin{cases}2, & i=j,\\ -1, & \{i, j\}=\{m-2, m\} \text { or }\{k, k+1\}, 1 \leq k \leq m-2, \\ 0, & \text { otherwise. }\end{cases}
\end{align*}
Note that since $\beta_i$ correspond to the positive simple roots of $\mathfrak{so}_{2m}$,
\begin{align*}
    a_{ij} = 2 \frac{(\beta_i,\beta_j)}{(\beta_i,\beta_i)}=(\beta_i,\beta_j)
\end{align*}
due to the properties of root systems.
\medskip

\subsubsection*{Diagram Automorphisms of $\mathfrak{so}_{2 m}$}

 \begin{definition}
 A diagram automorphism $\phi$ permutes the nodes of a Dynkin diagram while preserving its edge–vertex connectivity, i.e. $a_{i j}=a_{\phi (i) \phi (j)} $ for all $1 \leq i,j \leq m.$\end{definition} \medskip
 
 Let Aut($\mathfrak{g})$ denote the group of diagram automorphisms of a Lie Algebra $\mathfrak{g}$. We have that 
\begin{align*}
\begin{split}
    &\text{Aut}(D_m) \cong \mathbb{Z}/2 \quad \quad \quad \text{for } m \ge 5 \text{ and } m = 2,3\\
    &\text{Aut}(D_4) \cong S_3 \quad \ \quad \quad \text{The symmetric group on 3 elements}
\end{split}
\end{align*}

In this paper, our focus is on $\mathfrak{so}_{6}$ and $\mathfrak{so}_{8}$. In the first case, Aut($D_m$) corresponds to the transposition of the two final nodes $\{id, \,(m-1,m)\}$, while the special case of $m=4$ yields the automorphisms $\{id, \,(1,3), \,(1,4), \, (3,4), \, (1,3,4), \, (1,4,3)\}$. A diagrammatic illustration is shown below.

\begin{figure}[!h]
    \centering
    \includegraphics[width=8cm]{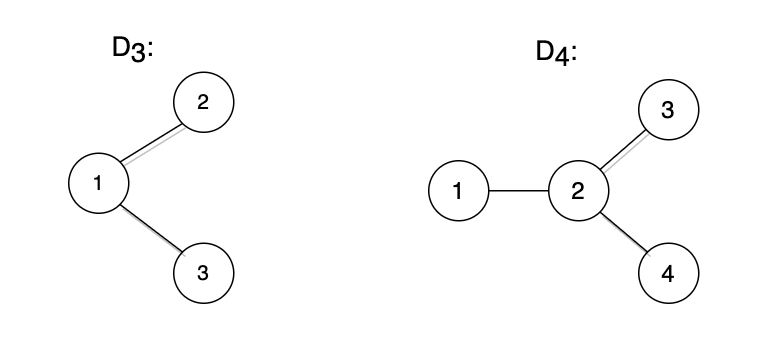}
    \label{fig:D_3}
\end{figure} 

 Let $\phi(i)$ be a diagram automorphism of $\mathfrak{so}_{2m}$ over $i = 1, \ldots,m$ such that $\phi ^2$=id. \\ 
The non-trivial possibilities for $\phi$ are:\\   $$\mathfrak{so}_{6}: (2,3); $$
$$\mathfrak{so}_{8}: (1,3),(1,4),(3,4).$$

\subsubsection{The quantized enveloping algebra $\mathcal{U}_{q}\left(\mathfrak{s o}_{2 m}\right)$}

 A primary motivation for using algebraic methods in probability stems from taking universal enveloping algebra's of an underlying Lie algebra $\mathfrak{g}$, and q-deforming it into the quantized enveloping algebra $\mathcal{U}_{q}\left( \mathfrak{g}\right)$.

 In this paper, we focus on the quantum group  $\mathcal{U}_{q}\left(\mathfrak{s o}_{2 m}\right)$, the quantized enveloping algebra generated by $\left\{E_{i}, F_{i}, q^{H_{i}}: 1 \leq i \leq m\right\}$ satisfying the following relations:
 \begin{equation}
\begin{split}
\label{UqSo2nRelations}
\left[E_{i}, F_{i}\right]=\frac{q^{H_{i}}-q^{-H_{i}}}{q-q^{-1}}, \\
q^{H_{i}} E_{j}=q^{\left(\beta_{i}, \beta_{j}\right)} E_{j} q^{H_{i}}, \quad \quad \quad  &q^{H_{i}} F_{j}=q^{-\left(\beta_{i}, \beta_{j}\right)} F_{j} q^{H_{i}},\\
E_{i}^{2} E_{j}+E_{j} E_{i}^{2}=(q+q^{-1}) E_{i} E_{j} E_{i}, \quad \quad &F_{i}^{2} F_{j}+F_{j} F_{i}^{2}=(q+q^{-1}) F_{i} F_{j} F_{i}.
\end{split}
\end{equation}

Note that we use $K_i:= q^{H_{i}}$ throughout the rest of the paper. \\

\subsubsection{Coproduct structure of $\mathcal{U}_{q}\left(\mathfrak{s o}_{2 m}\right)$}

We define the coproduct on $\mathcal{U}_{q}\left(\mathfrak{s o}_{2 m}\right)$ by specifying the coproduct on its generators: 
\begin{equation}
\begin{split}
\label{coprodDefinition}
    &\Delta: \mathcal{U}_{q}\left(\mathfrak{s o}_{2 m}\right) \to \mathcal{U}_{q}\left(\mathfrak{s o}_{2 m}\right) \otimes \mathcal{U}_{q}\left(\mathfrak{s o}_{2 m}\right) \text{ via: }\\
    \Delta(E_i) = E_i \otimes 1 + K_i &\otimes E_i, \quad \quad \quad  \Delta(F_i) = 1 \otimes F_i + F_i \otimes K_i^{-1}, \quad \quad \quad \Delta(K_i) = K_i \otimes K_i. 
\end{split}
\end{equation}
The coproduct is extended to the entire quantum group by defining it to be an algebra homomorphism.

 We will also need to define higher order powers of $\Delta$ from $\mathcal{U}_{q}\left(\mathfrak{s o}_{2 m}\right)$ to arbitrary tensor products of copies of $\mathcal{U}_{q}\left(\mathfrak{s o}_{2 m}\right)$. We again do so by specifying actions on the generators of $\mathcal{U}_{q}\left(\mathfrak{s o}_{2 m}\right)$:
\begin{equation}
\begin{split}
\label{iteratedCoprod}
    &\Delta^L: \mathcal{U}_{q}\left(\mathfrak{s o}_{2 m}\right) {\longrightarrow} \underbrace{\mathcal{U}_{q}\left(\mathfrak{s o}_{2 m}\right) \otimes \ldots \otimes \mathcal{U}_{q}\left(\mathfrak{s o}_{2 m}\right)}_{L+1 \text{ times}} \text{ via: }
    \\&\Delta^L(E_i) = \sum_{j=0}^L \underbrace{K_i \otimes \ldots \otimes K_i}_{j \text{ times}} \otimes E_i \otimes \underbrace{1 \otimes \ldots \otimes 1}_{L-j\text{ times}},
    \\&\Delta^L(F_i) = \sum_{j=0}^L \underbrace{1 \otimes \ldots \otimes 1}_{j \text{ times}} \otimes F_i \otimes \underbrace{K_i^{-1} \otimes \ldots \otimes  K_i^{-1}}_{L-j\text{ times}},
    \\&\Delta^L(K_i) = \underbrace{K_i \otimes \ldots \otimes K_i}_{L \text{ times}}.
\end{split}
\end{equation}

 An induction proof on $L$ shows that these indeed are the iterated coproducts on the generators in (\ref{coprodDefinition}).
\medskip

 We refer to the previous literature  \cite{Kassel} \cite{KlimykSchmudgen} for a general theory of quantum groups and their Hopf-algebra structures. In this paper, our focus is on the product, coproduct, and $*$-structures on $U_{q}\left(\mathfrak{s o}_{2 m}\right)$.\\

\subsubsection{Representation theory of $\mathcal{U}_{q}\left(\mathfrak{s o}_{2 m}\right)$}
Let  $\mathscr{E}_{i,j}$ denotes the indicator matrix with $1$ in entry $(i,j)$ and 0 elsewhere. \medskip


 The representations of $\mathcal{U}_{q}\left(\mathfrak{s o}_{2 m}\right)$ are built upon the fundamental representations of $\mathfrak{s o}_{2 m}$ from \cite{KLLPZ}. We define $\rho: \mathcal{U}_{q}\left(\mathfrak{s o}_{2 m}\right) \to \mathrm{Hom}(\mathbb{C}^{2m},\mathbb{C}^{2m}) $ as follows for all $1 \leq i \leq m$: 
\begin{equation}
\begin{split}
\label{UqSo2n Representation}
&\rho(E_{i})=\begin{cases}\mathscr{E}_{i, i+1}-\mathscr{E}_{m+i+1, m+i}, & 1 \leq i \leq m-1,\\ \mathscr{E}_{m-1,2 m}-\mathscr{E}_{m, 2 m-1}, & i=m,\end{cases}   \quad \quad \quad   \rho(F_{i})=\begin{cases}\mathscr{E}_{i+1, i}-\mathscr{E}_{m+i, m+i+1}, & 1 \leq i \leq m-1, \\ -\mathscr{E}_{2 m-1, m}+\mathscr{E}_{2 m, m-1}, & i=m,\end{cases}\\
\\&\rho(H_{i})=\begin{cases}\mathscr{E}_{i, i}-\mathscr{E}_{i+1, i+1}-\mathscr{E}_{m+i, m+i}+\mathscr{E}_{m+i+1, m+i+1}, & 1 \leq i \leq m-1, \\ \mathscr{E}_{m-1, m-1}+\mathscr{E}_{m, m}-\mathscr{E}_{2 m-1,2 m-1}-\mathscr{E}_{2 m, 2 m}, & i=m,\end{cases}\\
\\&\rho({K_i}) = q^{\rho(H_i)},   \qquad \qquad \qquad \quad \rho({K_i}^{-1}) = q^{-\rho(H_i)}. 
\end{split}
\end{equation}

\subsubsection{ $*-$structure on $\mathcal{U}_{q}\left(\mathfrak{s o}_{2 m}\right)$}
For later use, we introduce  explicit forms for the $*$--algebra structure on $U_{q}\left(\mathfrak{s o}_{2 m}\right)$. 

It can be checked easily that 
\item  \begin{equation}\label{eq: 1}
     K_i^*=K_{i}, \quad E_i^*=F_{i}, \quad F_i^*=E_{i}
\end{equation} gives a $*-$structure on $U_{q}\left(\mathfrak{s o}_{2 m}\right)$.

The $*-$structure is not unique,  there are more based on the automorphism $\phi$ of the Dynkin diagram of $U_{q}\left(\mathfrak{s o}_{2 m}\right)$ as defined in section \ref{diag}.

\begin{lemma}
The quantum groups $U_{q}\left(\mathfrak{s o}_{2 m}\right)$ is a Hopf $*$-algebra with 
$*:U_{q}\left(\mathfrak{s o}_{2 m}\right)  \rightarrow U_{q}\left(\mathfrak{s o}_{2 m}\right)$ defined as follows:
\begin{equation}\label{eq: 2}
     K_i^*=K_{\phi (i)}, \quad E_i^*=K_{\phi (i)}F_{\phi (i)}, \quad F_i^*=E_{\phi (i)}K^{-1}_{\phi(i)}.
\end{equation}

\end{lemma}
\begin{proof}


 It is also easy to see that the $*$ defined by \eqref{eq: 2} induces a $*-$structure with   $\phi$ outlined in section \ref{diag}, i.e. $*$ preserves $\mathcal{U}_{q}\left(\mathfrak{s o}_{2 m}\right)$ relations.
For instance, take the non-trivial $\phi$ for $\mathfrak{so}_{6}$ (i.e. $\phi=(2,3)$). Then,
$$K_2^*=K_3,\ E_2^*=K_3F_3,\ F_2^*=E_3K^{-1}_3,$$
\begin{align*}[E_{2}, F_{2}]^*=(\frac{K_{2}-K_{2}^{-1}}{q-q^{-1}})^*=\frac{K_{3}-K_{3}^{-1}}{q-q^{-1}}=[E_{3}, F_{3}]=E_3F_3-F_3E_3 \\ =E_3F_3-K_3F_3E_3K^{-1}_3=E_3K^{-1}_3K_3F_3-K_3F_3E_3K^{-1}_3\\
=[E_3K^{-1}_3,K_3F_3]=[F_{2}^*, E_{2}^*].
\end{align*}
Similarly, $*$ preserves all the other relations in \eqref{UqSo2nRelations}. Thus,  \eqref{eq: 2} gives a $*-$structure  on $U_{q}\left(\mathfrak{s o}_{2 m}\right)$. \\
\end{proof}
\begin{remark}
    It is conjectured that we can get different duality from different $*-$structures but we leave it to future work. In this paper, we will only make use of \eqref{eq: 1}.
\end{remark}
\section{Results}

\subsection{Orthogonal Polynomial Duality}
We first state some important facts for later use.
Define measure
$$
\mu_{\alpha}(\xi)=\prod_{x=1}^{L} \alpha^{\xi^{x}}q^{-2 x \xi^{x}},
$$
the reversible measure for type $D$  ASEP which arises from the ground state is given by $\nu(\eta)=\mu_{\alpha_1}(\eta_1)\mu_{\alpha_2}(\eta_2)$ (see Proposition 1.3 of \cite{KLLPZ}).

Define the $q$--hypergeometric function  
 ${_2\varphi_1}$  as:
\[
{ }_{2} \varphi_{1}\left(\begin{array}{c}
        a, b \\
        c
    \end{array} ;q, z\right):=\sum_{k=0}^{\infty} \frac{(a ; q)_{k}(b ; q)_{k}}{(c ; q)_{k}} \frac{z^{k}}{(q ; q)_{k}}
\]
and define the $q$-Krawtchouk polynomials  as:
\[
K_n(q^{-x},p,c,q) = { }_{2} \varphi_{1}\left(\begin{array}{c}
        q^{-x}, q^{-n} \\
        q^{-c}
    \end{array} ;q, pq^{n+1}\right).
\]



Then define function
\begin{equation}
    D^L_{\alpha_i}(\xi_i,\eta_i)=\prod_{x=1}^L K_{\eta_i^x}\left(q^{-2\xi_i^x}, p_i^x(\xi_i,\eta_i),1,q^2\right),
\end{equation}
where
\begin{equation}
    p_i^x(\xi_i,\eta_i)=\alpha_i^{-1}q^{-2\left(N_{x-1}^-(\xi_i)-N_{x+1}^+(\eta_i)\right)+2x-2}
\end{equation}
and
$\displaystyle{N_{x-1}^{-}\left(\xi_{i}\right)=\sum_{1\le y\le x-1} \xi_i^y} $ and $\displaystyle{N_{x+1}^{+}\left(\eta_{i}\right)=\sum_{x+1\le y\le L}\eta_i^y }$ denotes the number of particles in the configuration considered at the left, respectively right, of  site $x$.

$D_{\alpha_i}$ are orthogonal with respect to $\mu_{\alpha_i}$ when $\alpha_i \in(0, q^{-1+(2 L+1) })$  (see Theorem 3.2 in \cite{CFG} or Theorem 3.1 in \cite{FKZ}).

\begin{theorem}
Assume that $\alpha_1,\alpha_2 \in(0, q^{-1+(2 L+1) }).$ The type D ASEP with parameters $(q,2,0)$ or $(q,3,0)$  is self-dual with respect to the orthogonal function

\[D^L_{\alpha_1, \alpha_2}(\eta, \xi) = D^L_{\alpha_1}(\eta_1, \xi_1) \cdot D^L_{\alpha_2}(\eta_2, \xi_2).\]
\end{theorem}
\begin{remark}
    The duality function does not depend on the parameter $n$, it is believed that this Theorem still holds true for general $n$. However, we don't pursue this direction in this paper.
\end{remark}



\begin{proof}
We will actually provide two proofs for this theorem, with one being probabilistic and one being algebraic.

The probabilistic proof will be based on induction on the number of sites. For the base case of $L=2$, one simply explicitly computes the multiplication of $16\times 16$ matrices. Details are omitted. 

Note that $$K_{\eta_i^x}\left(q^{-2\xi_i^x}, p_i^x(\xi_i,\eta_i),1,q^2\right)\neq 1 \text{ if and only if  }\xi_i^x=\eta_i^x=1.$$
Rewrite the duality function as 
\begin{equation*}
    D^L_{\alpha_i}(\eta, \xi) = \prod_{x \in C(\eta_i, \xi_i)} \left( 1 - \frac{q^{2(x - N_{x-1}^{-}(\xi_i) + N_{x+1}^{+}(\eta_i)}}{\alpha_i} \right),
\end{equation*}
where $C(\eta_i, \xi_i) \subseteq \Lambda_L$ denotes the collection of locations of the common sites where both $\eta_i$ and $\xi_i$ have a particle of species $i$.

We proceed using induction on the number of sites, $L$, where we consider to 'append' a site to the right of a configuration. 
The duality with two sites can be checked directly. We assume that for any parameter $\alpha_1$ and $\alpha_2$,
\begin{equation}
\mathcal{L}^{L}D^L_{\alpha_1,\alpha_2}(\eta,\xi)=\mathcal{L}^{L}D^L_{\alpha_1,\alpha_2}(\eta,\xi) ,
\end{equation}where 
$\mathcal{L}^{L}$ acts on $\eta$ on the LHS and on $\xi$ on the RHS.
First,  note that the  generator on $L+1$ sites can be decomposed naturally into two sums: $\mathcal{L}^{L+1}=\mathcal{L}^{L}+\mathcal{L}^{L,L+1}$, where $\mathcal{L}^{L,L+1}$ is the local generator between sites $L$ and $L+1$, which gives 
\begin{equation*}
    \mathcal{L}^{L+1}D^{L+1}_{\alpha_1, \alpha_2}(\eta, \xi)=\mathcal{L}^{L}D^{L+1}_{\alpha_1, \alpha_2}(\eta, \xi)+\mathcal{L}^{L,L+1}D^{L+1}_{\alpha_1, \alpha_2}(\eta, \xi).
\end{equation*}

 Note that 
 \begin{equation}
     D^{L+1}_{\alpha_1, \alpha_2}(\eta, \xi)=D^{L}_{\alpha_1 q^{-2\eta^{L+1}_1}, \alpha_2 q^{-2\eta^{L+1}_2}}(\eta^{[1,L]}, \xi^{[1,L]})\prod_{i=1}^2K_{\eta_i^{L+1}}\left(q^{-2\xi_i^{L+1}}, p_i^{L+1}(\xi_i,\eta_i),1,q^2\right),
 \end{equation}
where $\eta^{[1,L]}$ is the restriction of $\eta$ to the first $L$ sites.
 
 \begin{multline}\label{eq: 18}
   \mathcal{L}^{L}D^{L+1}_{\alpha_1, \alpha_2}(\eta, \xi)=  \mathcal{L}^{L} D^{L}_{\alpha_1 q^{-2\eta^{L+1}_1}, \alpha_2 q^{-2\eta^{L+1}_2}}(\eta^{[1,L]}, \xi^{[1,L]})\prod_{i=1}^2K_{\eta_i^{L+1}}\left(q^{-2\xi_i^{L+1}}, p_i^{L+1}(\xi_i,\eta_i),1,q^2\right)  \\
   =\mathcal{L}^{L} D^{L}_{\alpha_1 q^{-2\eta^{L+1}_1}, \alpha_2 q^{-2\eta^{L+1}_2}}(\eta^{[1,L]}, \xi^{[1,L]})\prod_{i=1}^2K_{\eta_i^{L+1}}\left(q^{-2\xi_i^{L+1}}, p_i^{L+1}(\xi_i,\eta_i),1,q^2\right) = \mathcal{L}^{L}D^{L+1}_{\alpha_1, \alpha_2}(\eta, \xi),
 \end{multline}
 where in the first line, $\mathcal{L}$ acts on $\xi$ and in the second line, it acts on $\eta$. In addition, the induction hypothesis is used in the second equality of \eqref{eq: 18}.

 Last, factor $ D^{L+1}_{\alpha_1, \alpha_2}$ into the first $L-1$ sites and the last two sites, the proof reduces to the two site case:
    \begin{equation}
     D^{L+1}_{\alpha_1, \alpha_2}(\eta, \xi)=D^{2}_{\alpha_1 q^{2N_{L-2}^-{\xi_1}}, \alpha_2 q^{2N_{L-2}^-{\xi_2}}}(\eta^{[L-1,L]}, \xi^{[L-1,L]})\prod_{i=1}^{2}\prod_{x=1}^{L-1} K_{\eta_i^x}\left(q^{-2\xi_i^x}, p_i^x(\xi_i,\eta_i),1,q^2\right),
 \end{equation}   
  thus
  \begin{equation*}
      \mathcal{L}^{L,L+1}D^{L+1}_{\alpha_1, \alpha_2}(\eta, \xi)=\mathcal{L}^{L,L+1}D^{L+1}_{\alpha_1, \alpha_2}(\eta, \xi)
  \end{equation*}
        where $\mathcal{L}^{L,L+1}$ acts on $\eta$ on the LHS and on $\xi$ on the RHS. This finishes the probabilistic proof. 

  Now for the algebraic proof. We show the proof for the parameters $(q,2,0)$ only since the proof for $(q,3,0)$ follows from the same argument. First, we fix a choice of $*-$structure of $U_q(\mathfrak{so}_{6})$ as following:
\begin{equation}
     K_i^*=K_{i}, \quad E_i^*=F_{i}, \quad F_i^*=E_{i}.
\end{equation}

Notice that the subalgebra of $U_q(\mathfrak{so}_{2m})$ generated by $E_i,F_i$ and $K_i$ is isomorphic to $U_q(\mathfrak{sl}_2)$.

Then for $U_q(\mathfrak{so}_{6})$ with parameter $(q,2,0)$, apply the arguments from \cite{CFG} to the subalgebra generated by $E_2,F_2$ and $K_2$ we obtain a symmetry, which results in $D_{\alpha_1}(\eta_1, \xi_1)$, and the subalgebra generated by $E_3,F_3$ and $K_3$ gives $ D_{\alpha_2}(\eta_2, \xi_2)$.

First, we recall from \cite{KLLPZ} the particle configurations as the following:
\begin{center}
\begin{tikzpicture}
    \node at (-5,0){$v_2=$};
    \node[draw,fill=white] at (-4,-0.25){$1$};
    \node[draw,fill=white] at (-4,0.25){$2$};
    
    \node at (-2.5,0){$v_3=$};
    \node[draw,fill=white] at (-1.5,-0.25){$2$};
    \node[circle,draw,fill=white] at (-1.5,0.25){};
    
     \node at (0,0){$v_4=$};
    \node[draw,fill=white] at (1,-0.25){$1$};
    \node[circle,draw,fill=white] at (1,0.25){};
    
    \node at (2.5,0){$v_5=$};
    \node[circle,draw,fill=white] at (3.5,0.25){};
    \node[circle,draw,fill=white] at (3.5,-0.25){};
\end{tikzpicture}
\end{center}

Different from \cite{CFG}, $E_2$ acting on $v_5$ yields an additional $-1$, same as $F_2$ acting on $v_4$, $F_3$ acting on $v_3$ and $E_3$ acting on $v_5$. We will show that in the algebraic construction of duality, those extra $-1$s together give a constant.

Now we define two operator similar to  $S^{tr}_\alpha$ and $\hat{S}_\alpha^{tr}$ defined in section 8.3 of \cite{CFG}  as follows:
\begin{equation}
   S_1=e_{q^2}\left(-\sqrt{\alpha_1}(1-q^2)\Delta^{L-1}(K_{2}^{\frac{1}{2}}E_{2})\right), 
\end{equation}

\begin{equation}
   \hat{S}_1=\mathcal{E}_{q^2}\left(\sqrt{\alpha_1}q^{-L-\frac{1}{2}}(1-q^2)\Delta^{L-1}(K_{2}^{-\frac{1}{2}}E_{2})\right).
   \end{equation}

Then using equation (92)  in \cite{CFG}  and letting $N(\cdot)$ denote the total number of particles in a configuration, $\hat{S}_1S_1^*$ acts as the following:
\begin{equation}
   <\xi|\hat{S}_1S_1^*|\eta>=D_{\alpha_1}(\xi_1,\eta_1)q^{N(\eta_1)-N(\xi_1)}\sqrt{\mu_{\alpha_1}(\xi_1)\mu_{\alpha_1}(\eta_1)}(-1)^{N(\eta_1)}(-1)^{N_{v_4}(\xi)+N_{v_4}(\eta)},
\end{equation}
where 
$N_{v_4}(\xi)$ is the number of $v_4$ in configuration $\xi$, i.e  the number of sites containing both species particles.

Similarly, define
\begin{equation}
   S_2=e_{q^2}\left(-\sqrt{\alpha_2}(1-q^2)\Delta^{L-1}(K_{3}^{\frac{1}{2}}E_{3})\right), 
\end{equation}

\begin{equation}
   \hat{S}_2=\mathcal{E}_{q^2}\left(\sqrt{\alpha_2}q^{-L-\frac{1}{2}}(1-q^2)\Delta^{L-1}(K_{3}^{-\frac{1}{2}}E_{3})\right),
   \end{equation}
   We have
\begin{equation}
   <\xi|\hat{S}_2S_2^*|\eta>=D_{\alpha_2}(\xi_2,\eta_2)q^{N(\eta_2)-N(\xi_2)}\sqrt{\mu_{\alpha_2}(\xi_2)\mu_{\alpha_2}(\eta_2)}(-1)^{N(\eta_2)}(-1)^{N_{v_3}(\xi)+N_{v_3}(\eta)}.
\end{equation}

Note that 
\begin{equation*}
    (-1)^{N_{v_3}(\xi)+N_{v_4}(\xi)}=(-1)^{N_{v_3}(\xi)+N_{v_4}(\xi)+2N_{v_2}(\xi)}=(-1)^{N(\xi_1)+N(\xi_2)},
\end{equation*}
Thus 
\begin{equation}\label{eq: 3}   <\xi|\hat{S}_2S_2^*\hat{S}_1S_1^*|\eta>=\text{const}\cdot D_{\alpha_1}(\xi_1,\eta_1)D_{\alpha_2}(\xi_2,\eta_2)\sqrt{\mu_{\alpha_1}(\xi_1)\mu_{\alpha_1}(\eta_1)\mu_{\alpha_2}(\xi_2)\mu_{\alpha_2}(\eta_2)}.
\end{equation}
 By a standard argument as in \cite{CGRS,CFG,FKZ,KLLPZ}, dividing the symmetry \eqref{eq: 3} by square root of reversible measures i.e. $\sqrt{\nu(\xi)\nu(\eta)}$, we get the desired duality function.  

To show orthogonality, notice that the duality function is an independent product of $D_{\alpha_i}$, it is orthogonal with respect to the reversible measure $\nu$.

\end{proof}

\subsubsection{Conjecture about Asymptotics}
We conjecture that the type $D$ ASEP fluctuates as two Tracy--Widom distributions, based on the Markov self--duality and the reversible measures having the same form as for the usual ASEP.

Let $x_i(m,t)$ the position of the $m-$th particle of class $i$ from the left at time t, recall that when only one species of particles is present, the type $D$ ASEP with parameter $(q,n,0)$ reduces to a usual ASEP with jump rates
$$
q^{\pm 1}(q^{1-2n}+q^{2n-1}),
$$
with drift to the right for $q\in (0,1)$.

\textbf{Conjecture.} Let $\tau$ depend on $t,q$ and $n$ as
$$
\tau = t \cdot ( q^{1-2n}+q^{2n-1}) \cdot \frac{q+q^{-1}}{q-q^{-1}}.
$$
Then for the type $D$ ASEP with parameters $(2,0)$ and $(3,0)$ on the infinite lattice, and step initial conditions (all lattice sites to the left of $0$ are completely occupied with particles of both species), there is the asymptotic limit
\begin{equation}
    \lim_{t\xrightarrow{}\infty}\mathbb{P}\left(\frac{x_i(m_i,\tau)-c_{1i}t}{c_{2i}t^{1/3}}\le s\right)=F_2(s)
\end{equation}
where $\sigma_i=m_i/t$, $ c_{1i}=-1+2\sqrt{\sigma_i}$, $c_{2i}=\sigma_i^{-1/6}(1-\sqrt{\sigma_i})^{2/3}$. Here $F_2(\cdot)$ is the usual Tracy--Widom distribution.






\subsection{A more efficient approach than non--negativity}
To motivate the next section, we provide  more context about \cite{KLLPZ}. The paper constructs explicitly Casimir elements of $\mathcal{U}_q(\mathfrak{so}_6)$ and $\mathcal{U}_q(\mathfrak{so}_8)$ using Lusztig's inner product. However, these are very difficult to compute explicitly, effectively requiring one to invert a matrix with 8!=40320 rows and columns. In fact, the proof in \cite{KLLPZ} was computer aided with Python. After computing the Casimir, it was found that in order to maintain non--negativity of the jump rates, several states had to be ``discarded'', which is why there are only $4^L$ configurations, rather than $6^L$ or $8^L$. 

Because of the difficulty of computing Casimir elements, it would be helpful to develop a more efficient method to determine if and when states need to be discarded. In this paper, the duality functions are orthogonal with respect to the reversible measures, which suggests that studying the reversible measures is a more efficient approach than computing Casimir elements. This section will generalize Proposition 2.5 of \cite{KLLPZ}.  Note that \cite{KLLPZ} does not really prove Proposition 2.5 (it is simply described as a calculation with no details), so the proof here will also fill in the gaps in \cite{KLLPZ}.

Consider the following basis vectors:
\label{basisOfC6}
\begin{equation}
\begin{split}
\label{basisC6}
v_1=(1,0,0,0,0,0) \qquad &v_4=(0,0,0,0,0,1)\\
v_2=(0,1,0,0,0,0) \qquad &v_5=(0,0,0,0,1,0)\\
v_3=(0,0,1,0,0,0) \qquad &v_6=(0,0,0,1,0,0)
\end{split}
\end{equation}
 Associate to each vector a particle configuration by:
$$
v_{2}=\begin{array}{c}
\boxed{2}\\\boxed{1}
\end{array} \quad 
v_{3}=\begin{array}{c}
\bigcirc\\\boxed{2}
\end{array} \quad 
v_{4}=\begin{array}{c}
\bigcirc\\\boxed{1}
\end{array} \quad 
v_{5}=\begin{array}{c}
\bigcirc\\\bigcirc
\end{array} \quad 
$$

 Let $v_{5}$ be the vacuum state and let $\Omega_{L}=v_{5}^{\otimes L}$. \\

 Through direct computation, we can see that the actions of the operators $E$,$F$, and $K$ on the basis vectors are as follows:
\begin{itemize}
    \item $E_{2}$, $E_{3}$ create particles of class one and class two, respectively,
    \item $F_{2}, F_{3}$ annihilate particles of class one and two, respectively,
    \item $K_{2}, K_{3}$ do not change the vectors, but produce factors of q as coefficients. 
\end{itemize}

We start with $\Omega_{N}=v_{5}^{\otimes N}$. By acting with $E_3$ on $\Omega_{N}$, we introduce vector $v_3$ (since $E_3v_5=-v_3$). Next, the action of $E_2$ introduces vectors $v_2$ and $v_4$ (since $E_2v_3=v_2$ and $E_2v_5=-v_4$). \\\\
Since we wanted to extend this proposition to include a subsequent action with $E_1$, we had to come up with an associating particle configuration for $v_1$ (since $E_1v_2=v_1$).

We extend the above proposition from \cite{KLLPZ} so as to include the action of $E_1$. From a probabilistic perspective, fixing the vacuum vector $\Omega_{L}=v_{5}^{\otimes L}$, only the actions of $E_2$ and $E_3$ have probabilistic interpretations (creation of class 1 and class 2 particles). 
\medskip

 From an algebraic perspective, since we are working with the structure$\,\,\mathcal{U}_{q}\left(\mathfrak{s o}_{6}\right)$ with generators $E_1, E_2, E_3$, the question of how do the representations of $E_i$'s act on the basis of $\mathbb{C}^6$ should not exclude $E_1$. 

\begin{theorem}
For any $M_0, M_{1}, M_{2}, L \in \mathbb{N}$ such that $M_1, M_2 \leq L - M_0 \leq L$,
$$
E_1^{M_{0}}E_{2}^{M_{1}} E_{3}^{M_{2}}\left|\Omega_{L}\right\rangle=\sum_{\eta} \mathcal{G}(\eta)|\eta\rangle,
$$where the sum is over all particle configurations $\eta$ on $L$ sites with $M_{1}$ first class particles and $M_{2}$ second class particles, and $M_0$ null sites, and 
\begin{align*}
|\mathcal{G}(\eta)|= Z_{L, M_0, M_{1}, M_{2}}^{-1} \prod_{x_{0} \in (A_0 \cap A_{1} \cap A_2)^c(\eta)} q^{x_{0}} \prod_{x_{1} \in A_{1}(\eta)} q^{-x_{1}} \prod_{x_{2} \in A_{2}(\eta)} q^{-x_{2}}   \prod_{x_{3} \in A_{0}(\eta)} q^{-2x_{3}}
\end{align*}
for some normalization constant $Z_{L, M_0, M_{1}, M_{2}}$.
\end{theorem}
\medskip

\begin{remark}
    $A_0 \subseteq [L]$ are the locations of the null sites, $A_1 \subseteq [L]$ are the locations of the class 1 particles, $A_2 \subseteq [L]$ are the locations of the class 2 particles, $(A_0 \cup A_1 \cup A_2)^c \subseteq [L]$ are the sites with no particles (excluding null sites which will automatically have no particles).  
\end{remark}
\medskip

\begin{proof}
    We use the first q-exponential defined in Section \ref{qdn} to encode all possible values 
    of $L, M_0, M_1, M_2$. Equivalently, it suffices analyzing the following expression:
    \begin{equation}
        \label{propGeneralLHS}
        [\exp_{q^2}\Delta^{L-1}(E_1)][\exp_{q^2}\Delta^{L-1}(E_2)][\exp_{q^2}\Delta^{L-1}(E_3)]
        \left|\Omega_L\right\rangle.
    \end{equation}
    The q-exponential power series specifies all possible powers of its argument. So if we fix a particular $M_0, M_1, M_2$ as exponents of $E_1, E_2, E_3$, then the expression $E_1^{M_{0}}E_{2}^{M_{1}} E_{3}^{M_{2}}\left|\Omega_{L}\right\rangle$ can be found as a sub-series of equation (\ref{propGeneralLHS}). 
    
    \begin{lemma}
        We define the relevant actions of $E_i$ and $K_i$ on the basis (\ref{basisC6}), up to a factor of $-1$ since such factors will be cancelled out in $|\mathcal{G}(\eta)|$. 
        \begin{align*}
            &E_3 \text{ maps } v_5 \to v_3&&\\
            &E_2 \text{ maps } v_5 \to v_4  &&v_3 \to v_2\\
            &E_1 \text { maps } v_5 \to 0 &&v_2 \to v_1 \qquad v_3 \to 0 \quad v_4 \to 0\\
            &K_3 \text{ maps } v_5 \to q^{-1} v_5&&\\
            &K_2 \text{ maps } v_5 \to q^{-1}v_5  &&v_3 \to q^{-1}v_3\\
            &K_1 \text { maps } v_5 \to qv_5  &&v_2 \to q^{-1}v_2 \qquad v_3 \to v_3 \qquad v_4 \to v_4
        \end{align*}
  
    \end{lemma}   
    This lemma follows from direct computation on the matrix representations of $E_i$ and $K_i$.
     To approach equation (\ref{propGeneralLHS}), we first analyze the action $E_3^{M_2}$ on the vacuum vector $\Omega_L = v_5^{\otimes L}$.
    \begin{align}
        \label{InitialE3}
        [\exp_{q^2}\Delta^{L-1}(E_3)]\left|\Omega_L\right\rangle &= \exp_{q^2}
        (\sum_{j=0}^{L-1} \underbrace{K_3 \otimes \cdots \otimes K_3}_{j \text{ times}} \otimes E_3 \otimes \underbrace{1 \otimes \cdots \otimes 1}_{L-j-1\text{ times}})
        \left|\Omega_L\right\rangle
        \\\label{Pseudofactorization}
        &= \left[\prod_{j=0}^{L-1} \exp_{q^2}(\underbrace{K_3 \otimes \cdots \otimes K_3}_{j \text{ times}} \otimes E_3 \otimes \underbrace{1 \otimes \cdots \otimes 1}_{L-j-1\text{ times}})
        \right]\left|\Omega_L\right\rangle
        \\\label{Simplifying exponential}
        &= \left[\prod_{j=0}^{L-1} (1 + \underbrace{K_3 \otimes \cdots \otimes K_3}_{j \text{ times}} \otimes E_3 \otimes \underbrace{1 \otimes \cdots \otimes 1}_{L-j-1\text{ times}})
        \right]\left|v_5^{\otimes L}\right\rangle
        \\\label{Combinatorial argument}
        &= \sum_{k=0}^L \quad \sum_{S \subseteq [L],|S| = k} \quad \prod_{i=1}^k \, q^{-(S_i-1)} |V_S\rangle,
    \end{align}
    where $[L]$ denotes $\{1, 2, \ldots L\}$ 
    and the condition $|S|=k$ denotes a particle configuration with $k$ many class 2 particles, and $\{S_1,\ldots,S_k\}$ are the locations of those class $2$ particles, for example,
    \begin{align*}
        |V_S\rangle = v_5 \otimes v_3 \otimes v_3 \otimes v_5 \otimes v_3 
        \quad \text{ yields } S = \{S_1, S_2, S_3\} = \{2, 3, 5\}. 
    \end{align*}
    Equation (\ref{InitialE3}) to (\ref{Pseudofactorization}) follows from the Psuedo-factorization
    property of the q-exponential, as shown in \cite{CGRS}. The simplification to 
    (\ref{Simplifying exponential}) follows from the fact that higher order powers of $E_3$ are the 0 matrix 
    when viewed from the fundamental representation defined in Section 2.1.
    \medskip
    
     Finally, equation (\ref{Combinatorial argument})
    follows from a combinatorial argument. The binomial expansion of the product in (\ref{Simplifying exponential}) will lead to $2^L$ terms, which correponds to the sums over all subsets in (\ref{Combinatorial argument}). Every action of $K_3$ on $v_5$ produces a singular factor of $q^{-1}$, and distributing the product of operators in (\ref{Simplifying exponential}) yields that for every site $x \in [L]$ in which $E_3$ maps $v_5 \to v_3$, there will be $x-1$ appearances of $K_3$ and thus a factor of $q^{-(x-1)}$. This is best illustrated in a later concrete example. 
    \bigskip

     We now continue by examining the actions of $E_2^{M_1}$ on the resulting expression. 
    Since $q$ is a constant and operators are linear, we can work inside the nested sums/products and
    examine just the action on an existing particle configuration $|V_S\rangle$. 
    \begin{align}
        [\exp_{q^2}\Delta^{L-1}(E_2)]\left|V_S\right\rangle &= \exp_{q^2}
        (\sum_{j=0}^{L-1} \underbrace{K_2 \otimes \cdots \otimes K_2}_{j \text{ times}} \otimes E_2 \otimes \underbrace{1 \otimes \cdots \otimes 1}_{L-j-1\text{ times}})
        \left|V_S\right\rangle
        \\
        &= \left[\prod_{j=0}^{L-1} \exp_{q^2}(\underbrace{K_2 \otimes \cdots \otimes K_2}_{j \text{ times}} \otimes E_2 \otimes \underbrace{1 \otimes \cdots \otimes 1}_{L-j-1\text{ times}})
        \right]\left|V_S\right\rangle
        \\
        &= \left[\prod_{j=0}^{L-1} (1 + \underbrace{K_2 \otimes \cdots \otimes K_2}_{j \text{ times}} \otimes E_2 \otimes \underbrace{1 \otimes \cdots \otimes 1}_{L-j-1\text{ times}})
        \right]\left|V_S\right\rangle
        \\
        &= \sum_{k=0}^L \quad \sum_{I \subseteq [L],|I| = k} \quad \prod_{i=1}^k \, q^{-(I_i-1)} |V_{I,S}\rangle.
    \end{align}
    Using the same arguments of $E_3$'s action on $\Omega_L$. Again, as before, $|V_{I,S}\rangle$ denotes a particular particle configuration with a fixed many class 1 and class 2 particles. 
    For example, 
    \begin{align*}
        |V_{I,S}\rangle = v_5 \otimes v_3 \otimes v_2 \otimes v_4 \otimes v_2
        \quad \text{ yields } S = \{S_1, S_2, S_3\} = \{2, 3, 5\} \text{ and } I = \{I_1, I_2, I_3\} = \{3, 4, 5\}. 
    \end{align*}Note however,  $K_2$ and $K_3$ only produce 
    a factor $q^{-1}$ in both these cases, as $K_3v_5 = q^{-1}v_5$ and $K_2v_5 = q^{-1}v_5$
    and $K_2v_3 = q^{-1}v_3$. 
    \bigskip

     The action of $E_1$ on $|V_{I,S}\rangle$ is different from the analysis of $E_2$ and $E_3$:
    not only are there vectors being created that do not have a probabilistic particle interpretation
    (i.e. $v_1$) but also $K_1$ acts on $v_5$ producing a factor of $q$, while $K_1$ acts on $v_2$
    by producing a factor of $q^{-1}$. So
    in order to properly analyze the factors of $q$ produced, more attention must be paid to when $K_1$ acts, and on which vectors it acts upon. 
    \begin{equation*}
    \begin{split}
        [\exp_{q^2}\Delta^{N-1}(E_1)]\left|V_{I,S}\right\rangle &= \left[\prod_{j=0}^{N-1} (1 + \underbrace{K_1 \otimes \cdots \otimes K_1}_{j \text{ times}} \otimes E_1 \otimes \underbrace{1 \otimes \cdots \otimes 1}_{N-j-1\text{ times}})
        \right]\left|V_{I,S}\right\rangle\\
        & =\sum_{k=0}^N \quad \sum_{T \subseteq [N],|T| = k} \quad \prod_{i=1}^k \, q^{-[(2T_i - 2)-\lambda_1(T_i) - \lambda_2(T_i) + \lambda_5(T_i)]}\,\,\, |V_{T,I,S}\rangle, \\
    \end{split}
    \end{equation*}
 where $\lambda_n(x)$  is the number of  $v_n$  vectors from (\ref{basisC6}) strictly left of position $x$. 
    
The factors of $q$ produced depend entirely on which vectors $v_n$ that are acted upon by $K_1$. Most notably, for every null site $T_i$, this means that a $v_1$ ``particle'' must have emerged through $E_1v_2$. 
\medskip 

 Combinatorially, this means that there are a corresponding $T_i-1$ actions of $K_1$ on each of the sites to the left of $T_i$. The $\lambda_1$ and $\lambda_2$ terms take into account the action $K_2v_2$'s that occur to the left of $T_i$, each of which is $q^{-1}$. And the $\lambda_5$ term takes into account the $K_2v_5$'s that occur to the left of $T_i$, contributing a factor of $q$. Finally, the $(2T_i-2) = 2(T_i - 1)$ is present to preserve the original factors of $q$ that were acquired through $I$ and $S$ but no longer are multiplied when a double-occupied site becomes a null site. 
\bigskip
    
 Putting all three actions of the q-exponentials together, we can rewrite \eqref{propGeneralLHS} as the following:
\begin{equation}
\label{FinalSumGeneralize}
    \eqref{propGeneralLHS} = \sum_{M_0,M_1,M_2}\,\,\, \sum_{\eta}\left[\prod_i^{|S|}q^{-(S_i -1)}\right]\left[\prod_j^{|I|}q^{-(I_j -1)}\right]\left[\psi\right] \,\, \left|V_{T, I,S}\right\rangle   
\end{equation} 
 where $ \psi = -((2T_i - 2)-\lambda_1(T_i) - \lambda_2(T_i) + \lambda_5(T_i))$.

The outer sum is over all $M_0, M_1, M_2$ properly bounded in $N$ such that $M_1, M_2 \leq N - M_0$, and the inner sum is over all $\eta$ with null site, class 1, and class 2 locations as sets $T, I, S$ respectively, such that $|T| = M_0$, $|I| = M_1$, and $|S| = M_2$. The notation follows conventions used directly prior. 
\bigskip 

 This decomposition (\ref{FinalSumGeneralize}) shows that $E_1^{M_{0}}E_{2}^{M_{1}} E_{3}^{M_{2}}\left|\Omega_{N}\right\rangle=\sum_{\eta} \mathcal{G}(\eta)|\eta\rangle$, a sum of particle configurations scaled by $\mathcal{G(\eta})$. Moreover, we can express this scaling $|\mathcal{G}(\eta)|$ in a factorizable form: 
\begin{align*}
    |\mathcal{G}(\eta)|= Z_{N, M_0, M_{1}, M_{2}}^{-1} \prod_{x_{0} \in (A_0 \cap A_{1} \cap A_2)^c(\eta)} q^{x_{0}} \prod_{x_{1} \in A_{1}(\eta)} q^{-x_{1}} \prod_{x_{2} \in A_{2}(\eta)} q^{-x_{2}}   \prod_{x_{3} \in A_{0}(\eta)} q^{-2x_{3}}
\end{align*}
up to a normalization constant $Z$, by looking at the factors of $q$ that are contributed by null sites, class 1 and class 2 particles, and empty sites. Each of the $v_5$ empty sites contributes positive factors of $q$, the class 1 and class 2 sites contribute negative exponents of $q$, and null sites emulate factors of $q^{-2}$, which come from the $-2T_i$ found in $\psi$. This completes the proof. 
\end{proof}

\bibliographystyle{alpha}
\bibliography{main}

\end{document}